\RequirePackage{fix-cm}
\documentclass[smallextended]{svjour3}       
\smartqed  
\usepackage{graphicx}
\usepackage{amsmath}
\usepackage{amsfonts}
\usepackage{amssymb}
\usepackage{dsfont}

\newtheorem{thm}{Theorem}
\newtheorem{col}{Corollary}
\newtheorem{lem}{Lemma}

 \journalname{}
\begin{document}

\title{Probabilistic Stirling numbers of the second kind and applications
}


\author{Jos\'e A. Adell}


\institute{Jos\'e A. Adell \at
              Departamento de M\'etodos Estad\'{i}sticos. Facultad de Ciencias\\ Universidad de Zaragoza\\
                Pedro Cerbuna 12, 50009 Zaragoza, Spain \\
              \email{adell@unizar.es}           
}

\date{Received: date / Accepted: date}

\maketitle

\begin{abstract}
Associated to each complex-valued random variable satisfying appropriate integrability conditions, we introduce a different generalization of the Stirling numbers of the second kind. Various equivalent definitions are provided. Attention, however, is focused on applications. Indeed, such numbers describe the moments of sums of i.i.d. random variables, determining their precise asymptotic behavior without making use of the central limit theorem. Such numbers also allow us to obtain explicit and simple Edgeworth expansions. Applications to L\'{e}vy processes and cumulants are discussed, as well.
\keywords{Probabilistic Stirling number\and moment \and Edgeworth expansion \and L\'{e}vy process \and cumulant \and generalized difference}
 \subclass{60E05\and 05A19}
\end{abstract}

\section{Introduction}\label{Sec1}

The classical Stirling numbers play an important role in many branches of mathematics and physics as ingredients in the computation of diverse quantities. In particular, the Stirling numbers of the second kind $S(j,m)$, counting the number of partitions of $\{1,\ldots ,j\}$ into $m$ non-empty, pairwise disjoint subsets, are a fundamental tool in many combinatorial problems. Such numbers can be defined in various equivalent ways (cf.  Abramowitz and Stegun \cite[p. 824]{AS} and Comtet \cite[Chap. 5]{C}). Two of the most useful are the following. Let $j=0,1,\ldots$ and $m=0,1,\ldots ,j$. Then $S(j,m)$ can be explicitly defined as
\begin{equation}\label{e100}
S(j,m)=\dfrac{1}{m!}\sum_{k=0}^m \binom{m}{k}(-1)^{m-k}k^j,
\end{equation}
or via their generating function as
\begin{equation}\label{e105}
\dfrac{(e^z-1)^m}{m!}=\sum_{j=m}^\infty S(j,m)\dfrac{z^j}{j!},\quad z\in \mathds{C}.
\end{equation}

Motivated by various specific problems, different generalizations of the Stirling numbers $S(j,m)$ have been considered in the literature (see, for instance, Hsu and Shiue \cite{HS}, Luo and Srivastava \cite{LU}, Caki\'{c} \textit{et al.} \cite{CA}, and El-Desouky \textit{et al.} \cite{EL}, among many others). In \cite{AL3}, we considered the following probabilistic generalization. Let $(Y_k)_{k\geq 1}$ be a sequence of independent copies of a real-valued random variable $Y$ having a finite moment generating function and denote by $S_k=Y_1+\cdots +Y_k, k=1,2,\ldots$ $(S_0=0)$. Then, the Stirling numbers of the second kind associated to $Y$ are defined by
\begin{equation}\label{e110}
S_Y(j,m)=\dfrac{1}{m!}\sum_{k=0}^m \binom{m}{k}(-1)^{m-k}\mathds{E}S_k^j,\quad j=0,1,\ldots,\quad m=0,1,\ldots ,j.
\end{equation}
Observe that formula (\ref{e110}) recovers (\ref{e100}) when $Y=1$. The motivations behind definition (\ref{e110}) have to do with certain problems coming from analytic number theory, such as extensions in various ways of the classical formula for sums of powers on arithmetic progressions (cf. \cite{AL3}) and explicit expressions for higher order convolutions of Appell polynomials (see \cite{AL4}).

In this paper, we extend definition (\ref{e110}) to complex-valued random variables $Y$ and show its usefulness in various classical topics of probability theory. In this regard, we show in Section \ref{Sec3} that the moments $\mathds{E}S_n^j$ can be written in closed form in terms of the Stirling numbers $S_Y(j,m)$. When $Y$ is real-valued and centered, two remarkable consequences deserve to be mentioned. First, we can directly obtain the precise asymptotic behavior of $\mathds{E}S_n^j$ as far as rates of convergence and leading coefficients are concerned, without appealing to the central limit theorem. Monotonicity properties of the sequence $\mathds{E}S_n^{2j}/(n)_j, n\geq j$, where $(n)_j=n(n-1)\cdots (n-j+1)$, are also derived in a simple way. We point out that monotonicity results in the central limit theorem seem to be rather scarce (in this respect, Teicher \cite{TE} and Kane \cite{KA} showed that $P(S_n\geq 0)$ converges monotonically for various choices of the law of $S_n$). Second, from a computational point of view, we can evaluate $\mathds{E}S_n^j$ for $n\geq j$ in terms of $\mathds{E}S_n^j$ for $n<j$ and $S_Y(j,\lfloor j/2\rfloor)$. In Section \ref{Sec4}, we deal with analogous properties referring to L\'{e}vy processes and centered subordinators.

Concerning rates of convergence in the central limit theorem, Edgeworth expansions provide a great accuracy in the approximation at the price of using rather involved technicalities (cf. Petrov \cite{PE}, Hall \cite{HA}, Barbour \cite{BA}, and Rinot and Rotar \cite{RI}, among others). In Section \ref{Sec5}, we give explicit and relatively simple full Edgeworth expansions whose coefficients depend on the Stirling numbers $S_{Y+iZ}(j,m)$, where the real-valued random variables $Y$ and $Z$ are independent and $Z$ has the standard normal distribution. The order of magnitude of such expansions is that of $n^{-(r-1)/2}$, whenever $\mathds{E}Y^k=\mathds{E}Z^k, k=1,2,\ldots,r$, for some $r\geq 2$. In Section \ref{Sec6}, we show that the cumulants of a random variable $Y$ can also be described by means of $S_Y(j,m)$. Finally, in Section \ref{Sec2} we gather some equivalent definitions of $S_Y(j,m)$ when $Y$ is complex-valued without proofs, since they are similar to those previously given in \cite{AL3} for real-valued random variables $Y$.

\section{Probabilistic Stirling numbers}\label{Sec2}

The following notations will be used throughout the paper. Let $\mathds{N}$ be the set of positive integers and $\mathds{N}_0=\mathds{N}\cup \{0\}$. Unless otherwise specified, we assume that $j,m\in\mathds{N}_0$, $x\in\mathds{C}$, and $z\in\mathds{C}$ satisfies $|z|<R$, where $R>0$ may change from line to line. We always consider measurable exponentially bounded functions $f:\mathds{C}\rightarrow \mathds{C}$, i.e., $|f(x)|\leq e^{R|x|}$. We denote by $I_j(x)=x^j$ the jth monomial function and by $(x)_j$ the descending factorial, that is, $(x)_j=x(x-1)\cdot\cdot\cdot(x-j+1), j\in \mathds{N}$, $(x)_0=1$. Finally, we set $j\wedge m= \min(j,m)$ and denote by $\lfloor y\rfloor$ the integer part of $y\in \mathds{R}$.

Let $\mathcal{G}_0$ be the set of complex-valued random variables $Y$ having a finite moment generating function in a neighborhood of the origin, i.e.,
\begin{equation*}
\mathds{E}e^{|zY|}<\infty,\quad  |z|<R,
\end{equation*}
for some $R>0$.

For any $r\in \mathds{N}$, we consider a random variable $\beta(r)$ having the beta density
\begin{equation}\label{e115}
\rho_r(\theta)=r(1-\theta)^{r-1},\quad 0\leq \theta \leq 1,
\end{equation}
whereas we set $\beta(0)=1$. Note that $\beta(1)$ is uniformly distributed on $[0,1]$. For any $r\in \mathds{N}_0$, let $(Y_k)_{k\geq 1}$ and $(\beta_k(r))_{k\geq 1}$ be two sequences of independent copies of $Y\in \mathcal{G}_0$ and $\beta(r)$, respectively, and assume that both sequences are mutually independent. We denote
\begin{equation}\label{e120}
W_m(r,Y)= \beta_1(r)Y_1+\cdots+\beta_m(r)Y_m,\quad m\in\mathds{N}\quad (W_0(r,Y)=0).
\end{equation}
The following two important special cases will also be denoted
\begin{equation}\label{e125}
W_m(0,Y)=Y_1+\cdots +Y_m=S_m,\qquad W_m(2,Y)=W_m(Y),\quad m\in \mathds{N}_0.
\end{equation}

On the other hand, consider the difference operator
\begin{equation*}
\triangle_y^1f(x)=f(x+y)-f(x),\quad y\in \mathds{C},
\end{equation*}
together with the iterates
\begin{equation}\label{e130}
\triangle_{y_1,\ldots,y_m}^m f(x)=(\triangle_{y_1}^1\circ\cdot\cdot\cdot\circ\triangle_{y_m}^1)f(x),\quad (y_1,\ldots,y_m)\in\mathds{C}^m,\quad m\in\mathds{N}.
\end{equation}
Such generalized difference operators were used by Mrowiec et al. \cite{MR} and Dilcher and Vignat \cite{DI} in different analytical contexts. Observe that
\begin{equation*}
\triangle^m f(x):=\triangle_{1,\ldots,1}^m f(x)=\sum_{k=0}^m \binom{m}{k}(-1)^{m-k}f(x+k),\quad m\in \mathds{N}_0,
\end{equation*}
is the usual $m$th forward difference of $f$. In general, the iterates in (\ref{e130}) have a cumbersome expression. However, we have the following formulas stated in \cite{AL3}, where it is understood that
\begin{equation*}
\triangle_\emptyset^m f(x)=f(x),\qquad \prod_{k=1}^0=1.
\end{equation*}

\begin{lem}\label{l000}
Let $Y\in\mathcal{G}_0$. For any $m\in\mathds{N}_0$, we have
\begin{equation*}
\mathds{E}\triangle_{Y_1,\ldots ,Y_m}^m f(x)=\sum_{k=0}^m \binom{m}{k}(-1)^{m-k}\mathds{E}f(x+S_k).
\end{equation*}
If, in addition, $f$ is $m$ times differentiable, then
\begin{equation*}
\mathds{E}\triangle_{Y_1,\ldots ,Y_m}^m f(x)=\mathds{E}Y_1\cdot\cdot\cdot Y_m f^{(m)}(x+W_m(1,Y)).
\end{equation*}
\end{lem}

The Stirling numbers of the second kind $S_Y(j,m), m\leq j$, associated to the random variable $Y\in\mathcal{G}_0$ are defined as in (\ref{e110}). Observe that this definition is justified in the sense that
\begin{equation}\label{e135}
S_Y(j,m)=0,\quad m>j,
\end{equation}
as follows by choosing $f=I_j$ and $x=0$ in Lemma \ref{l000}. Such numbers are characterized in the following result (cf.\cite{AL3}).
\begin{thm}\label{t100}
Let $Y\in\mathcal{G}_0$. For any $m\leq j$, we have
\begin{equation}\label{e140}
\begin{split}
S_Y(j,m)&=\frac{1}{m!}\mathds{E}\triangle_{Y_1,\ldots ,Y_m}^m I_j(0)=\binom{j}{m}\mathds{E}Y_1\cdot\cdot\cdot Y_m W_m(1,Y)^{j-m}\\
&= \frac{1}{m!}\sum_{l=0}^j S(j,l)\sum_{k=0}^m \binom{m}{k}(-1)^{m-k}\mathds{E}(S_k)_l.
\end{split}
\end{equation}

Equivalently, the numbers $S_Y(j,m)$ are defined via their generating function as
\begin{equation}\label{e145}
\frac{1}{m!}\left(\mathds{E}e^{zY}-1\right)^m=\sum_{j=m}^\infty S_Y(j,m)\frac{z^j}{j!}.
\end{equation}
\end{thm}

For the classical Stirling numbers $S(j,m)$, expression (\ref{e140}) gives us
\begin{equation*}
S(j,m)=\binom{j}{m}\mathds{E}(\beta_1(1)+\cdots +\beta_m(1))^{j-m},\quad m\leq j.
\end{equation*}
This formula was already obtained by Sun \cite{SU}.

Theorem \ref{t100} allows us to obtain explicit expressions of $S_Y(j,m)$ for different choices of the random variable $Y$ (see \cite{AL3}). In many cases, such numbers are actually real numbers. For instance, if $Y=U+iV$, where $U$ and $V$ are independent real-valued random variables and $V$ has a real characteristic function (in particular, if $V=0$ or if $V$ has the standard normal distribution). In fact, let $t\in\mathds{R}$ with $|t|<R$. Since $\mathds{E}e^{tY}=\mathds{E}e^{tU}\mathds{E}e^{itV}$ is real, we see that both sides in (\ref{e145}) are real when choosing $z=t$. This shows the claim. Finally, if $Y$ is nonnegative, then $S_Y(j,m)$ is nonnegative as well, as follows from (\ref{e120}) and (\ref{e140}).

\section{Moments}\label{Sec3}

In this section, we give closed form expressions for the moments of $S_n$, as defined in (\ref{e125}) in terms of the probabilistic Stirling numbers $S_Y(j,m)$ and discuss some of their consequences. In this respect, for any $r\in \mathds{N}$, denote by $\mathcal{G}_r$ the subset of $\mathcal{G}_0$ consisting of those random variables $Y=U+iV$ such that
\begin{equation}\label{e150}
\mathds{E}Y^k=0,\quad k=1,\ldots,r.
\end{equation}
In the case in which $U$ and $V$ are independent, observe that $Y\in \mathcal{G}_1$, if $\mathds{E}U=\mathds{E}V=0$; $Y\in \mathcal{G}_2$ if, in addition, $\mathds{E}U^2=\mathds{E}V^2$; $Y\in \mathcal{G}_3$ if, moreover, $\mathds{E}U^3=\mathds{E}V^3=0$; $Y\in \mathcal{G}_4$ if, additionally,
\begin{equation}\label{e155}
\mathds{E}U^4-\binom{4}{2}(\mathds{E}U^2)^2+\mathds{E}V^4=0,
\end{equation}
and so on. Also observe that if $U$ and $V$ are independent copies of a random variable having the standard normal distribution, then $Y\in \mathcal{G}_r$, for any $r\in \mathds{N}$, since
\begin{equation}\label{e160}
\mathds{E}e^{itY}=1,\quad t\in \mathds{R}.
\end{equation}
The following auxiliary result will be used in Section \ref{Sec5}.
\begin{lem}\label{l005}
Let $r\in \mathds{N}$ and let $U$ and $V$ be two independent real-valued random variables such that $U\in \mathcal{G}_0$ and $V$ has the standard normal distribution. Then, $Y=U+iV\in \mathcal{G}_r$ if and only if $\mathds{E}U^k=\mathds{E}V^k, k=1,\ldots ,r$.
\end{lem}

\begin{proof}
Let $Z$ be an independent copy of $V$. By (\ref{e160}), we see that
\begin{equation}\label{e165}
\mathds{E}(Z+iV)^l=0,\quad l\in \mathds{N}.
\end{equation}
Assume that $\mathds{E}U^k=\mathds{E}V^k, k=1,\ldots,r$, and let $s=1,\ldots,r$. By (\ref{e165}), we have
\begin{equation*}
\begin{split}
\mathds{E}(U+iV)^s&=\sum_{k=0}^s \binom{s}{k}\mathds{E}U^k\mathds{E}(iV)^{s-k}\\
&=\sum_{k=0}^s \binom{s}{k}\mathds{E}Z^k\mathds{E}(iV)^{s-k}=\mathds{E}(Z+iV)^s=0.
\end{split}
\end{equation*}
To show the reverse implication, we use induction on $r$. For $r=1$, the result is obviously true. Assume that the result is true for some $r\in \mathds{N}$. Let $Y\in \mathcal{G}_{r+1}\subseteq \mathcal{G}_r$. By the induction assumption, $\mathds{E}U^k=\mathds{E}V^k=\mathds{E}Z^k$, $k=1,\ldots ,r$. We thus have from (\ref{e165})
\begin{equation*}
\begin{split}
0&=\mathds{E}(U+iV)^{r+1}=\mathds{E}U^{r+1}+\sum_{k=0}^r \binom{r+1}{k}\mathds{E}Z^k\mathds{E}(iV)^{r+1-k}\\
&=\mathds{E}U^{r+1}+\mathds{E}(Z+iV)^{r+1}-\mathds{E}Z^{r+1}=\mathds{E}U^{r+1}-\mathds{E}Z^{r+1}=\mathds{E}U^{r+1}-\mathds{E}V^{r+1}.
\end{split}
\end{equation*}
This shows the reverse implication and completes the proof.
\end{proof}

The interesting feature of the random variables $Y$ in the subset $\mathcal{G}_r, r\in \mathds{N}$, is that its corresponding Stirling numbers satisfy $S_Y(j,m)=0$, for $j<m(r+1)$, as shown in the following result. This property has remarkable consequences to evaluate the moments $\mathds{E}S_n^j$, as seen in the remaining results of this section, as well as to obtain the Edgeworth expansions considered in Section \ref{Sec5}.
\begin{thm}\label{t105}
Let $Y\in\mathcal{G}_r$, for some $r\in \mathds{N}_0$. Then,
\begin{equation*}
S_Y(j,m)=\frac{(m(r+1))!}{m!((r+1)!)^m}\binom{j}{m(r+1)}\mathds{E}(Y_1\cdot\cdot\cdot Y_m)^{r+1} W_m(r+1,Y)^{j-m(r+1)},
\end{equation*}
whenever $j\geq m(r+1)$, whereas $S_Y(j,m)=0$, if $j< m(r+1)$.
\end{thm}

\begin{proof}
We start with the following identity, which follows from the formula for the remainder term in Taylor's theorem:
\begin{equation*}
e^z=\sum_{k=0}^r \frac{z^k}{k!}+\frac{z^{r+1}}{(r+1)!}\mathds{E}e^{z\beta(r+1)},
\end{equation*}
where $\beta(r+1)$ is the random variable defined in (\ref{e115}). Replacing $z$ by $zY$ in this formula and then taking expectations, we have from (\ref{e120}) and (\ref{e150})
\begin{equation*}
\begin{split}
\frac{1}{m!}\left(\mathds{E}e^{zY}-1\right)^m &=\frac{z^{m(r+1)}}{m!((r+1)!)^m}\left(\mathds{E}Y^{r+1}e^{z\beta(r+1)Y}\right)^m\\
&=\frac{z^{m(r+1)}}{m!((r+1)!)^m}\mathds{E}(Y_1\cdot\cdot\cdot Y_m)^{r+1}e^{zW_m(r+1,Y)}\\
&=\frac{z^{m(r+1)}}{m!((r+1)!)^m}\sum_{k=0}^\infty \mathds{E}(Y_1\cdot\cdot\cdot Y_m)^{r+1}W_m(r+1,Y)^k\frac{z^k}{k!}.
\end{split}
\end{equation*}
Thus, the result follows from (\ref{e145}) with the change $j=m(r+1)+k$.
\end{proof}

\begin{thm}\label{t110}
Let $Y\in\mathcal{G}_r$, for some $r\in \mathds{N}_0$. Denote by $\tau_r(j)=\lfloor j/(r+1)\rfloor$. For any $n\in \mathds{N}_0$, we have
\begin{equation}\label{e170}
\mathds{E}S_n^j=\sum_{m=0}^{n\wedge \tau_r(j)}S_Y(j,m)(n)_m.
\end{equation}
Moreover, for any $n\geq \tau_r(j)$, we have
\begin{equation}\label{e175}
\begin{split}
\frac{\mathds{E}S_n^j}{(n)_{\tau_r(j)}}
= & S_Y(j,\tau_r(j))\\
& +\frac{1}{(\tau_r(j)-1)!}\sum_{k=0}^{\tau_r(j)-1}\binom{\tau_r(j)-1}{k}\frac{(-1)^{\tau_r(j)-k-1}}{n-k}\mathds{E}S_k^j.
\end{split}
\end{equation}
\end{thm}

\begin{proof}
Note that
\begin{equation}\label{e180}
\left(\mathds{E}e^{zY}\right)^n=\mathds{E}e^{zS_n}=\sum_{j=0}^\infty \mathds{E}S_n^j\frac{z^j}{j!}.
\end{equation}
By (\ref{e145}) and Theorem \ref{t105}, we see that
\begin{equation*}
\begin{split}
\left(\mathds{E}e^{zY}\right)^n&=\sum_{m=0}^n (n)_m \frac{\left(\mathds{E}e^{zY}-1\right)^m}{m!}=\sum_{m=0}^n (n)_m \sum_{j=m(r+1)}^\infty S_Y(j,m)\frac{z^j}{j!}\\
&=\sum_{j=0}^\infty \frac{z^j}{j!}\sum_{m=0}^{n\wedge \tau_r(j) } S_Y(j,m)(n)_m.
\end{split}
\end{equation*}
This, together with (\ref{e180}), shows (\ref{e170}). On the other hand, if $n=\tau_r(j)$, formula (\ref{e175}) directly follows from definition (\ref{e110}). Assume that $n>\tau_r(j)$. The following combinatorial identity
\begin{equation*}
\sum_{l=0}^p \binom{s}{l}(-1)^l=\binom{s-1}{p}(-1)^p,\qquad p,s\in\mathds{N}_0,\quad p\leq s-1,
\end{equation*}
can be easily shown by induction on $p$. Using (\ref{e110}), (\ref{e170}), and the preceding identity, we have
\begin{equation*}
\begin{split}
\mathds{E}S_n^j&=\sum_{m=0}^{\tau_r(j)}\binom{n}{m}\sum_{k=0}^m \binom{m}{k}(-1)^{m-k}\mathds{E}S_k^j\\
&=\sum_{k=0}^{\tau_r(j)}\binom{n}{k}\mathds{E}S_k^j\sum_{m=k}^{\tau_r(j)}\binom{n-k}{m-k}(-1)^{m-k}\\
&=\sum_{k=0}^{\tau_r(j)}\binom{n}{k}\mathds{E}S_k^j\binom{n-k-1}{\tau_r(j)-k}(-1)^{\tau_r(j)-k}\\
&=\binom{n}{\tau_r(j)}\sum_{k=0}^{\tau_r(j)}\binom{\tau_r(j)}{k}\frac{n-\tau_r(j)}{n-k}(-1)^{\tau_r(j)-k}\mathds{E}S_k^j.
\end{split}
\end{equation*}
This, together with definition (\ref{e110}), shows (\ref{e175}). The proof is complete.
\end{proof}

The classical Stirling numbers of the second kind $S(j,m)$ can also be defined by means of the equations
\begin{equation}\label{e181}
x^j=\sum_{m=0}^j S(j,m)(x)_m,\quad j\in\mathds{N}_0.
\end{equation}
In this sense, formula (\ref{e170}) may be thought as the probabilistic counterpart of (\ref{e181}).
\begin{col}\label{c099}
Let $Y\in\mathcal{G}_0$. Then,
\begin{equation*}
|S_Y(j,m)|\leq \frac{\mathds{E}(|Y_1|+\cdots +|Y_m|)^j}{m!},\quad 0\leq m \leq j.
\end{equation*}
\end{col}

\begin{proof}
Let $0\leq m \leq j$. By (\ref{e120}) and the second equality in (\ref{e140}), we see that $|S_Y(j,m)|\leq S_{|Y|}(j,m)$. Applying (\ref{e170}) with $r=0$, we get
\begin{equation*}
\mathds{E}(|Y_1|+\cdots +|Y_m|)^j=\sum_{k=0}^m S_{|Y|}(j,k)(m)_k \geq S_{|Y|}(j,m)m!.
\end{equation*}
This completes the proof.
\end{proof}

This result extends the well known upper bound for the classical Stirling numbers of the second kind, namely,
\begin{equation*}
S(j,m)\leq \frac{m^j}{m!},\quad 0\leq m \leq j.
\end{equation*}

The case in which $Y\in\mathcal{G}_1$ is real-valued deserves special attention. First, denote by $\sigma^2=\mathds{E}Y^2$ its variance and define the real-valued random variable $\widetilde{Y}$ whose distribution function is given by
\begin{equation*}
    F_{\widetilde{Y}}(y)=\dfrac{1}{\sigma^2}\int_{-\infty}^y x^2 F_Y(dx),\quad y\in \mathds{R},
\end{equation*}
where $F_Y$ is the distribution function of $Y$ and it is assumed that $\sigma^2>0$. Note that for any function $f$ we have
\begin{equation}\label{e185}
\sigma^2\mathds{E}f(\widetilde{Y})=\mathds{E}Y^2f(Y).
\end{equation}
In the trivial case in which $Y=0$, a.s., we define $\widetilde{Y}=0$, a.s., so that formula \eqref{e185} still holds. Second, consider the random variable $W_m(\widetilde{Y})$ as defined in (\ref{e125}). Finally, recall that if $Z$ is a random variable having the standard normal distribution, then
\begin{equation}\label{e190}
\mathds{E}Z^{2m}=\frac{(2m)!}{m!2^m},\quad m\in \mathds{N}_0.
\end{equation}

With these ingredients, we give the following.
\begin{col}\label{c100}
Let $Y\in\mathcal{G}_1$ be real-valued. Then,
\begin{equation}\label{e191}
S_Y(j,m)=\binom{j}{2m}\mathds{E}(\sigma Z)^{2m}\mathds{E}W_m(\widetilde{Y})^{j-2m},
\end{equation}
whenever $j\geq 2m$, whereas $S_Y(j,m)=0$, for $j<2m$. In particular,
\begin{equation}\label{e192}
S_Y(2j,j)=\mathds{E}(\sigma Z)^{2j},\qquad S_Y(2j+1,j)=j(2j+1)\mathds{E}(\sigma Z)^{2j}\frac{\mathds{E}Y^3}{3\sigma^2}.
\end{equation}
\end{col}

\begin{proof}
The proof of (\ref{e191}) follows along the lines of that of Theorem \ref{t105}, by taking into account (\ref{e190}) and the fact that
\begin{equation*}
\mathds{E}e^{zY}-1=\frac{z^2}{2}\mathds{E}Y^2e^{z\beta(2)Y}=\frac{\sigma^2 z^2}{2}\mathds{E}e^{z\beta(2)\widetilde{Y}},
\end{equation*}
as follows from (\ref{e185}). The identities in (\ref{e192}) are a consequence of (\ref{e191}) and the equalities
\begin{equation*}
\mathds{E}\beta(2)=\frac{1}{3}, \qquad \mathds{E}\widetilde{Y}=\frac{\mathds{E}Y^3}{\sigma^2}.
\end{equation*}
The proof is complete.
\end{proof}

\begin{col}\label{c105}
Let $Y\in\mathcal{G}_1$ be real-valued. Then,
\begin{equation}\label{e193}
\mathds{E}S_n^j=\sum_{m=0}^{n\wedge \lfloor j/2\rfloor}S_Y(j,m)(n)_m.
\end{equation}
As a consequence, the sequence $\mathds{E}S_n^{2j}/(n)_j, n\geq j$, decreases to $\mathds{E}(\sigma Z)^{2j}$.

Moreover, for any $n\geq j$, we have
\begin{equation}\label{e194}
\begin{split}
\frac{\mathds{E}S_n^j}{(n)_{\lfloor j/2\rfloor}} = &S_Y(j,\lfloor j/2\rfloor)\\
&+\frac{1}{(\lfloor j/2\rfloor-1)!}\sum_{k=0}^{\lfloor j/2\rfloor-1}\binom{j-1}{k}\frac{(-1)^{\lfloor j/2\rfloor-k-1}}{n-k}\mathds{E}S_k^j.
\end{split}
\end{equation}
\end{col}

\begin{proof}
As follows from (\ref{e191}), the Stirling numbers $S_Y(2j,m)$ are positive. This, together with (\ref{e192}) and (\ref{e193}), implies that the sequence $\mathds{E}S_n^{2j}/(n)_j, n\geq j$, decreases to $S_Y(2j,j)=\mathds{E}(\sigma Z)^{2j}$. The remaining assertions readily follow from Theorem \ref{t110} by choosing $r=1$. The proof is complete.
\end{proof}

Let $Y\in\mathcal{G}_1$ be real-valued. Traditionally, the problem of convergence concerning the moments $\mathds{E}S_n^j$, as $n\rightarrow \infty$, is carried out by establishing first the central limit theorem
\begin{equation*}
\frac{S_n}{\sigma \sqrt{n}}\to Z,\quad n\to \infty,
\end{equation*}
and afterwards showing (see, for instance, von Bahr \cite{VO}) that
\begin{equation*}
\mathds{E}\left(\frac{S_n}{\sigma \sqrt{n}}\right)^j\to \mathds{E}Z^j,\quad n\to \infty.
\end{equation*}
The explicit expressions in Corollaries \ref{c100} and \ref{c105} directly give us the precise asymptotic behaviour of the moments $\mathds{E}S_n^j$ as far as rates of convergence and leading coefficients are concerned. Note, in particular, that the odd moments $\mathds{E}S_n^{2j+1}$ have the order of magnitude of $(n)_j$ (resp. $(n)_{j-1}$) with leading coefficients $S_Y(2j+1,j)$ (resp. $S_Y(2j+1,j-1)$) in the case that $\mathds{E}Y^3\neq 0$ (resp. $\mathds{E}Y^3=0$), as follows from (\ref{e192}).

On the other hand, $\mathds{E}S_n^{2j}/(n)_j$ decreasingly converges to $\mathds{E}(\sigma Z)^{2j}$. This monotonicity property is no longer true, in general, for the odd moments, since the leading coefficient of $\mathds{E}S_n^{2j+1}/(n)_j$ depends on $\mathds{E}Y^3$. Another consequence of formula (\ref{e194}) is that, with the help of (\ref{e192}), we can quickly compute the moments $\mathds{E}S_n^j$ for any $n\geq j$ in terms of the corresponding moments for $n<j$.

\section{L\'{e}vy processes and centered subordinators}\label{Sec4}

L\'{e}vy processes are, in continuous time, the analogue to sums of independent identically distributed random variables in discrete time. It therefore seems plausible to obtain for such processes similar moment results to those given in the preceding section. Recall that a L\'{e}vy process $(Y(t))_{t\geq 0}$ is a stochastically continuous process starting at the origin and having independent stationary increments. A subordinator $(X(t))_{t\geq 0}$ is a L\'{e}vy process having right-continuous nondecreasing paths.

Let $(W(t))_{t\geq 0}$ be a zero mean square integrable L\'{e}vy process whose characteristic function is given by (cf. Steutel and van Harn \cite[p. 181]{ST})
\begin{equation}\label{eqA}
    \mathds{E}e^{i\xi W(t)}=\exp \left ( t\int_\mathds{R} \left ( e^{i\xi x}-1-i\xi x\right )\, K(dx)\right ),\quad \xi\in \mathds{R},\quad t\geq 0,
\end{equation}
where $K(dx)$ is a L\'{e}vy measure on $\mathds{R}$, which puts mass 0 on $\{0\}$ and satisfies
\begin{equation*}
    \int_\mathds{R}|x|\, K(dx)<\infty,\quad 0<\kappa^2=\int_\mathds{R}x^2\, K(dx)<\infty.
\end{equation*}
The characteristic function \eqref{eqA} can be written as
\begin{equation}\label{eqB}
    \mathds{E}e^{i\xi W(t)}=\exp \left ( -\dfrac{t\kappa^2 \xi^2}{2}\mathds{E}e^{i\xi \beta (2)U}\right ),\quad \xi\in \mathds{R},\quad t\geq 0,
\end{equation}
where $\beta(2)$ is defined in (\ref{e115}) and $U$ is a random variable independent of $\beta(2)$, with distribution function
\begin{equation*}
    F_U(x)= \dfrac{1}{\kappa^2} \int_{-\infty}^x y^2\, K(dy),\quad x\in \mathds{R}.
\end{equation*}
We see from \eqref{eqB} that $\mathds{E}W(t)=0$ and $\mathds{E}W^2(t)=t\kappa^2$, $t\geq 0$, so that $\kappa^2$ is the variance of $W(1)$.

Now, let $(B(t))_{t\geq 0}$ be a standard Brownian motion on $\mathds{R}$, independent of $(W(t))_{t\geq 0}$ and define the L\'{e}vy process $(Y(t))_{t\geq 0}$ by setting
\begin{equation}\label{eqC}
    Y(t)=W(t)+\sigma B(t),\quad t\geq 0,\quad \sigma \geq 0.
\end{equation}
Observe that $\mathds{E}Y(t)=0$, $\mathds{E}Y^2(t)=t(\kappa^2 +\sigma^2)$, $t\geq 0$, and
\begin{equation}\label{eqD}
    \mathds{E}e^{i\xi Y(t)}=\exp \left ( -\dfrac{t\sigma^2 \xi^2}{2}-\dfrac{t\kappa^2 \xi^2}{2} \mathds{E}e^{i\xi \beta(2) U}\right ),\quad \xi \in \mathds{R},\quad t\geq 0,
\end{equation}
as follows from \eqref{eqB} and \eqref{eqC}. Let $V$ be a random variable uniformly distributed on $[0,1]$ and independent of $\beta(2)$ and $U$. Then, we can rewrite \eqref{eqD} as
\begin{equation}\label{eqE}
    \begin{split}
    \mathds{E}e^{i\xi Y(t)} &= \exp \left ( -\dfrac{t(\sigma^2+\kappa^2)\xi^2}{2}\, \mathds{E}\exp \left ( i\xi \beta(2) U 1_{\{V< \kappa^2/(\sigma^2+\kappa^2)\}}\right )\right )\\
    &=\exp \left ( -\dfrac{t(\sigma^2+\kappa^2)\xi^2}{2} \, \mathds{E}e^{i\xi \beta(2)T_\star}\right ),\quad \xi \in \mathds{R},\quad t\geq 0,
    \end{split}
\end{equation}
where
\begin{equation}\label{eqF}
    T_\star=U 1_{\{V< \kappa^2/(\sigma^2+\kappa^2)\}}.
\end{equation}

On the other hand, a subordinator $(X(t))_{t\geq 0}$ is called centered if $\mathds{E}(X(t)-t)=0$ and $\mathds{E}(X(t)-t)^2<\infty, t\geq 0$. In such a case, the characteristic function of $X(t)$ is then given by (cf. Steutel and van Harn \cite[p. 107]{ST} and \cite{AL1})
\begin{equation*}
\mathds{E}e^{i\xi X(t)}=\exp\left(t\mathds{E}\frac{e^{i\xi T}-1}{T}\right),\quad \xi\in \mathds{R},\quad t\geq 0,
\end{equation*}
where $T$ is a nonnegative random variable. Denote by $\tau^2=\mathds{E}T$. This notation comes from the fact that $\mathds{E}(X(t)-t)^2=t\tau^2, t\geq 0$. Consider the nonnegative random variable $T^\ast$ whose distribution function is given by
\begin{equation*}
    F_{T^\ast} (y)=\dfrac{1}{\tau ^2} \int_0^y xF_T(dx), \quad y\geq 0,
\end{equation*}
and equal to zero for $y<0$, where $F_T$ is the distribution function of $T$ and it is assumed that $\tau^2>0$. In the case in which $T=0$, a.s., we simply define $T^\ast=0$, a.s. Observe that for any function $f$ we have
\begin{equation*}
\tau^2\mathds{E}f(T^\ast)=\mathds{E}Tf(T).
\end{equation*}
It turns out that (cf. \cite{AL1})
\begin{equation}\label{e197}
\mathds{E}e^{i\xi(X(t)-t)}=\exp\left(-\frac{t\tau^2\xi^2}{2}\mathds{E}e^{i\xi \beta(2)T^\ast}\right),\quad \xi\in \mathds{R},\quad t\geq 0,
\end{equation}
where the random variables $T^\ast$ and $\beta(2)$ are independent. The main difference between formulas \eqref{eqE} and (\ref{e197}) is that $T_\star$ is real-valued, whereas $T^\ast$ is nonnegative. We finally observe that if $(X(t))_{t\geq 0}$ is the standard Poisson process, then $T=T^\ast=1$, whereas for the gamma process, the random variables $T$ and $T^\ast$ have the probability densities $\rho(\theta)=e^{-\theta}$ and $\rho^\ast(\theta)=\theta e^{-\theta}, \theta\geq 0$, respectively.

Once representations \eqref{eqE} and (\ref{e197}) are given, we can obtain closed form expressions for the moments of $Y(t)$ and $X(t)-t$ in a simple way, as the following result shows.
\begin{thm}\label{t115}
Assume that $T_\star$ and $T^\ast$, appearing in \eqref{eqF} and (\ref{e197}), respectively, belong to $\mathcal{G}_0$. For any $j\in \mathds{N}_0$ and $t\geq 0$, we have
\begin{equation*}
g_j(t):=\frac{\mathds{E}Y(t)^j}{t^{\lfloor j/2\rfloor}}=\sum_{m=0}^{\lfloor j/2\rfloor}\binom{j}{2m}(\sigma^2+\kappa^2)^m\mathds{E}Z^{2m}
\mathds{E}W_m(T_\star)^{j-2m}\frac{1}{t^{\lfloor j/2\rfloor-m}},
\end{equation*}
and
\begin{equation*}
h_j(t):=\frac{\mathds{E}(X(t)-t)^j}{t^{\lfloor j/2\rfloor}}=\sum_{m=0}^{\lfloor j/2\rfloor}\binom{j}{2m}\mathds{E}(\tau Z)^{2m}\mathds{E}W_m(T^\ast)^{j-2m}\frac{1}{t^{\lfloor j/2\rfloor-m}}.
\end{equation*}

Moreover, the functions $g_{2j}(t)$ and $h_j(t)$ are completely monotonic.
\end{thm}

\begin{proof}
The identities in Theorem \ref{t115} follow by expanding the characteristic functions given in \eqref{eqE} and (\ref{e197}), and recalling (\ref{e190}). The last statements concerning complete monotonicity follow from the facts that $\mathds{E}W_m(T_\star)^{2(j-m)}$ and $\mathds{E}W_m(T^\ast)^{j-2m}$ are nonnegative for $m\leq \lfloor j/2\rfloor$. The proof is complete.
\end{proof}

With respect to Theorem \ref{t115}, similar comments to those made after Corollary \ref{c105} are valid. Details are omitted.

\section{Edgeworth expansions}\label{Sec5}

Let $y\in \mathds{R}$ and let $Z$ be a random variable having the standard normal density
\begin{equation*}
g(y)=\frac{1}{\sqrt{2\pi}}e^{-y^2/2}.
\end{equation*}
Denote by $G(y)$ the standard normal distribution function. Recall that the Hermite polynomials $(H_n(y))_{n\geq 0}$ are defined by
\begin{equation*}
g(y)H_n(y)=(-1)^n g^{(n)}(y).
\end{equation*}
Since
\begin{equation*}
\frac{1}{2\pi}\int_\mathds{R}\mathds{E}e^{i\zeta(Z-y)}d\xi=\frac{1}{2\pi}\int_\mathds{R}e^{-i\xi y}e^{-\xi^2/2}d\xi=\frac{1}{\sqrt{2\pi}}\mathds{E}e^{-iyZ}=g(y),
\end{equation*}
differentiation under the integral sign with respect to $y$ gives us
\begin{equation}\label{e200}
\frac{1}{2\pi}\int_\mathds{R}(i\xi)^n\mathds{E}e^{i\zeta(Z-y)}d\xi=(-1)^n g^{(n)}(y)=g(y)H_n(y).
\end{equation}

Let $Y\in \mathcal{G}_0$ be a real-valued random variable having an integrable characteristic function. Suppose that $\mathds{E}Y=0$ and $\mathds{E}Y^2=1$. Denote by $F_n(y)$ the distribution function of $S_n/\sqrt{n}$. Under such circumstances, it is well known (see, for instance, Petrov \cite[p. 117]{PE}) that
\begin{equation}\label{e205}
F_n(y)-G(y)=-\frac{1}{2\pi}\int_\mathds{R} e^{-i\xi y}\frac{\mathds{E}e^{i\xi S_n/\sqrt{n}}-\mathds{E}e^{i\xi Z}}{i\xi} d\xi.
\end{equation}

We will show that the Edgeworth expansion of $F_n(y)-G(y)$ can be described in a simple way in terms of the Stirling numbers associated to the complex-valued random variable
\begin{equation}\label{e210}
\widehat{Y}=Y+iZ,
\end{equation}
where $Y$ and $Z$ are supposed to be independent. To this end, fix $r=2,3,\ldots $ Consider the sets
\begin{equation*}
\Delta=\{(m,j): 1\leq m\leq n,\quad j\geq m(r+1)\},
\end{equation*}
and
\begin{equation}\label{e215}
\Delta_k=\{(m,j)\in \Delta: j=2m+k\},\quad k=r-1,r,r+1,\ldots
\end{equation}

We are in a position to state the following.
\begin{thm}\label{t200}
Let $Y\in \mathcal{G}_0$ be a real-valued random variable having an integrable characteristic function. Assume that $\mathds{E}Y^k=\mathds{E}Z^k, k=1,2,\ldots ,r$ for some $r\geq 2$. For any $n\in \mathds{N}$ and $y\in \mathds{R}$, we have
\begin{equation}\label{e220}
F_n(y)-G(y)=-g(y)\sum_{k=r-1}^\infty \frac{1}{n^{k/2}}\sum_{(m,j)\in \Delta_k}\frac{S_{\widehat{Y}}(j,m)}{j!}\frac{(n)_m}{n^m}H_{j-1}(y).
\end{equation}
\end{thm}

\begin{proof}
Let $\xi\in \mathds{R}$. By (\ref{e210}), the integrand in (\ref{e205}) can be written as
\begin{equation}\label{e225}
\mathds{E}e^{i\xi (Z-y)}\frac{\left(\mathds{E}e^{i\xi \widehat{Y}/\sqrt{n}}\right)^n-1}{i\xi}.
\end{equation}
By Lemma \ref{l005}, the random variable $\widehat{Y}$ belongs to $\mathcal{G}_r$. We therefore have from Theorem \ref{t105} and (\ref{e215})
\begin{equation*}
\begin{split}
\frac{\left(\mathds{E}e^{i\xi \widehat{Y}/\sqrt{n}}\right)^n-1}{i\xi}&= \frac{1}{i\xi}\sum_{m=1}^n (n)_m \frac{\left(\mathds{E}e^{i\xi \widehat{Y}/\sqrt{n}}-1\right)^m}{m!}\\
&= \frac{1}{i\xi}\sum_{m=1}^n (n)_m \sum_{j=m(r+1)}^\infty \frac{S_{\widehat{Y}}(j,m)}{j!}\left(\frac{i\xi}{\sqrt{n}}\right)^j\\
&=\sum_{k=r-1}^\infty \frac{1}{n^{k/2}}\sum_{(m,j)\in\Delta_k}\frac{S_{\widehat{Y}}(j,m)}{j!}\frac{(n)_m}{n^m}(i\xi)^{j-1}.
\end{split}
\end{equation*}
Hence, integrating (\ref{e225}) with respect to $\xi$, the conclusion follows from (\ref{e200}). The proof is complete.
\end{proof}

Fix $r\geq 2$. Compared with the usual full Edgeworth expansions, Theorem \ref{t200} gives us an explicit and relatively simple expansion of $F_n(y)-G(y)$, making clear, at the same time, that its order of magnitude is that of $n^{-(r-1)/2}$, provided that the first moments of $Y$ and $Z$ (up to the order $r\geq 2$) coincide. The coefficients in this expansion depend on the Stirling numbers $S_{\widehat{Y}}(j,m)$ associated to the complex-valued random variable $\widehat{Y}$ defined in (\ref{e210}). As noted after Theorem \ref{t100}, such numbers are actually real numbers, which can be evaluated by means of Theorem \ref{t105}.

For instance, let us evaluate the leading coefficient in (\ref{e220}). As follows from (\ref{e215}), we have $\Delta_{r-1}=\{(1,r+1)\}$. Thus, the leading coefficient in (\ref{e220}) is equal to
\begin{equation*}
-g(y)H_r(y)\frac{S_{\widehat{Y}}(r+1,1)}{(r+1)!}=-g(y)H_r(y)\frac{\mathds{E}\widehat{Y}^{r+1}}{(r+1)!},
\end{equation*}
as follows from Theorem \ref{t105}. On the other hand, if $r+1$ is odd, it can be checked from (\ref{e210}) that $\mathds{E}\widehat{Y}^{r+1}=\mathds{E}Y^{r+1}$, whereas if $r+1=2s$, we have from (\ref{e190}) and the moment assumptions in Theorem \ref{t200}
\begin{equation*}
\begin{split}
\mathds{E}\widehat{Y}^{r+1}&=\mathds{E}\widehat{Y}^{2s}=\sum_{k=0}^{2s}\binom{2s}{k}\mathds{E}Y^k\mathds{E}(iZ)^{2s-k}\\
&= \mathds{E}Y^{2s}+\sum_{l=0}^{s-1}\binom{2s}{2l}\mathds{E}Z^{2l}(-1)^{s-l}\mathds{E}Z^{2(s-l)}\\
&= \mathds{E}Y^{2s}+\frac{(2s)!}{s!2^s}\sum_{l=0}^{s-1}\binom{s}{l}(-1)^{s-l}=\mathds{E}Y^{2s}-(-1)^s\mathds{E}Z^{2s}.
\end{split}
\end{equation*}

\section{Cumulants}\label{Sec6}

Recall that the cumulant generating function of a random variable $Y\in \mathcal{G}_0$ is defined as
\begin{equation}\label{e300}
K_Y(z)=\log \mathds{E}e^{zY}=\sum_{j=1}^\infty \kappa_j(Y)\frac{z^j}{j!},
\end{equation}
where the coefficients $(\kappa_j(Y))_{j\geq 1}$ are called the cumulants of $Y$. Such cumulants can be written in terms of the Stirling numbers $S_Y(j,m)$, as shown in the following result.
\begin{thm}\label{t300}
Let $Y\in \mathcal{G}_0$. For any $j\in \mathds{N}$, we have
\begin{equation}\label{e305}
\kappa_j(Y)=\sum_{m=1}^j (-1)^{m-1}(m-1)!S_Y(j,m)=\sum_{k=1}^j \binom{j}{k}\frac{(-1)^{k-1}}{k}\mathds{E}S_k^j.
\end{equation}
\end{thm}

\begin{proof}
Using the expression
\begin{equation*}
\log(1+x)=\sum_{m=1}^\infty (-1)^{m-1}\frac{x^m}{m},\quad |x|<1,
\end{equation*}
and choosing $z$ in a neighborhood of the origin so that $|\mathds{E}e^{zY}-1|<1$, we get
\begin{equation*}
\begin{split}
K_Y(z)&=\log \left(1+\mathds{E}e^{zY}-1\right)=\sum_{m=1}^\infty (-1)^{m-1}(m-1)!\frac{\left(\mathds{E}e^{zY}-1\right)^m}{m!}\\
&=\sum_{m=1}^\infty (-1)^{m-1}(m-1)!\sum_{j=m}^\infty S_Y(j,m)\frac{z^j}{j!}\\
&=\sum_{j=1}^\infty \frac{z^j}{j!}\sum_{m=1}^j(-1)^{m-1}(m-1)!S_Y(j,m).
\end{split}
\end{equation*}
In view of (\ref{e300}), this shows the first equality in (\ref{e305}). The second one readily follows from definition (\ref{e110}) and the well known combinatorial identity
\begin{equation*}
\sum_{l=0}^p \binom{s+l}{l}=\binom{s+p+1}{p},\quad p,s\in \mathds{N}_0.
\end{equation*}
The proof is complete.
\end{proof}

We finally mention that, in certain particular cases, we find for the cumulants simpler formulas than those given in (\ref{e305}). For instance, in the case of the L\'{e}vy processes considered in \eqref{eqE}, it can be checked that
\begin{equation*}
\kappa_j(Y(t))=t(\sigma^2+\kappa^2)\mathds{E}T_\star^{j-2},\quad j=2,3,\ldots,\quad t\geq 0.
\end{equation*}
A similar formula holds for the centered subordinators defined in (\ref{e197}).

\begin{acknowledgements}
We thank the referees for their careful reading of the manuscript and for their comments and suggestions, which greatly improved the final outcome.

This work was partially supported by Ministerio de Ciencia, Innovaci\'on y Universidades, Project PGC2018-097621-B-I00.
\end{acknowledgements}

%
\section*{Conflict of interest}

 The author declares that he has no conflict of interest.



\end{document}